\documentclass[letterpaper, 10 pt, conference]{ieeeconf}  

\IEEEoverridecommandlockouts                              

\title{\LARGE \bf
Simulation of Nonlinear Systems Trajectories: between Models and Behaviors*
}

\author{Antonio Fazzi$^{1}$ 
	and Alessandro Chiuso$^{1}$
\thanks{*This work was not supported by any organization}
\thanks{$^{1}$the authors are with the Department of Information Engineering, 
        University of Padova, via Gradenigo 6/b, I-35131 Padova, Italy
        {\tt\small antonio.fazzi@unipd.it}, 
    {\tt\small alessandro.chiuso@unipd.it}
}%
}

\newcommand{\R}{\mathbb R}

\usepackage{amsmath,amsfonts,amssymb,color}
\usepackage{url}
\newtheorem{theorem}{Theorem}
\newtheorem{definition}[theorem]{Definition}
\newtheorem{lemma}[theorem]{Lemma}

\newtheorem{problem}[theorem]{Problem}
\newtheorem{remark}[theorem]{Remark}

\usepackage{algorithm}
\usepackage{algpseudocode}
\usepackage{graphics}
\usepackage{graphicx}
\usepackage{epstopdf}

\begin{document}

\maketitle
\thispagestyle{empty}
\pagestyle{empty}

\begin{abstract}
	In this paper, we study connections between the classical model-based approach to nonlinear system theory, where systems are represented by equations, and the nonlinear behavioral approach, where systems are defined as sets of trajectories. In particular, we focus on equivalent representations of the systems in the two frameworks for the problem of simulating a future nonlinear system trajectory starting from a given set of noisy data. The goal also includes extending some existing results from the deterministic to the stochastic setting. 


\end{abstract}

\section{INTRODUCTION}
\label{sec:intro}
The classical way to solve a problem in the field of system theory is to use a (possibly estimated) system model, that is a set of equations describing the dynamic, to implement the sought design. The difficulties with this approach started  whenever the to be modeled systems became complex and~/~or interconnected, hence the model estimation step became the most computationally expensive part of the whole algorithmic procedure \cite{expdesign}.  To deal with this issue, motivated by the increasing quantity of data, the popularity of \textit{data-driven} algorithms increased; such methods design the solution method by learning the system dynamic directly from a set of data (without using explicitly the system equations). 

A nowadays popular data-driven approach is the \textit{behavioral setting} \cite{PoldermanWillems}, where dynamical systems are defined as sets of trajectories (the observed data). Modern applications motivated the extension of such a theory from linear time-invariant (LTI) to classes of nonlinear systems \cite{narx,rueda2020data,ddsim-narx,Hemelhof}. Such extensions are realized by a LTI embedding of the nonlinear dynamic over a suitable finite dimensional space.  

After a brief introduction and a recap of related results, we focus on the problem of simulating trajectories of some nonlinear systems, by looking at the problem from both the data-driven viewpoint and the model-based one. The main contribution to this problem is twofold: propose an iterative data-driven prediction algorithm whose solution is the same as the model-based approach, and extend some existing results from the deterministic to the stochastic setting. 

 The problem of predicting future system trajectories is the cornerstone of predictive control, thus our analysis provides the basis
  for nonlinear predictive control, whose solution is currently based on linearization (e.g. \cite{BerbdatadrivenMPC}).

 \section{Notation and preliminaries}
%
%
The definition of a (discrete-time and real valued) dynamical system is a triple $(\mathbb{N}, \mathbb{R}^q, \mathcal{B})$, where $\mathbb{N}$ is the time axis, $\mathbb{R}^q$ is the signal space ($q$ is the number of variables) and $\mathcal{B} \subseteq  (\mathbb{R}^q)^{\mathbb{N}}$ is a subset of admissible functions $w: \mathbb{N} \to \mathbb{R}^q$, the \emph{behavior}. The word \emph{admissible} means that the trajectories $w \in \mathcal{B}$ satisfy a  difference equation.

A leading role in the behavioral system theory is played by Hankel matrices built from observed trajectories since such matrices are able to \emph{learn} (under suitable assumption, see Lemma \ref{lemma:gpe} below) the system dynamic.
\begin{definition}
	\label{def:hank}
	Given a time series $w(t) \in \R^q$, the Hankel matrix with $L$ rows, $H_L(w) \in \R^{qL \times (T-L+1)}$, is
	\begin{equation}
	H_L(w) = \begin{pmatrix}
	w(1) & w(2) & \cdots & w(T-L+1) \\
	w(2) & w(3) & \cdots & w(T-L+2) \\
	\vdots & \vdots &  & \vdots \\
	w(L) & w(L+1) & \cdots & w(T) 
	\end{pmatrix}
	\end{equation}.
\end{definition}
\begin{lemma}[\cite{Ident}]
	\label{lemma:gpe}
	We are given a LTI system $\mathcal{B}$ with $q$ variables, $m$ inputs and order $n$, and one of its length-$T$ trajectories $w \in \mathcal{B}$.   
	It holds that rank $H_L(w) = mL+n$ if and only if $\mathcal{B} =$ image $H_L(w)$. 
\end{lemma}

Lemma \ref{lemma:gpe} is a powerful result since it allows to write down all the system trajectories from an observed one. However, the LTI assumption on the system limits its applications. 
On the same line of \cite{ddsim-narx}, we consider single-input single-output (SISO) nonlinear systems of the form:
\begin{equation}
\label{nonleq}
\begin{aligned}
\mathcal{B}_{nl} &= \{w = \begin{bmatrix}
u \\ y
\end{bmatrix} \in (\mathbb{R}^2)^{N} | \  \eqref{nonleq}\ \  \text{holds} \}, \\
y(t) &= f(x_y(t), x_{u}(t), t), \ \ t=1, \dots, N \\
x_y(t) &= (\sigma^{-\ell}y(t), \dots, \sigma^{-1}y(t))^T, \\
x_{u}(t) = &(\sigma^{-\ell}u(t), \dots, \sigma^{-1}u(t), u(t))^T, \\
f(x_y, x_{u}) &= 
  \theta_{lin} \begin{bmatrix}
x_y \\ x_{u}
\end{bmatrix} + \theta_{nl} \phi_{nl}(x_y, x_{u}). 
\end{aligned}
\end{equation}
In \eqref{nonleq} $u, y$ are the input (the free variable) and the output (the dependent variable), respectively, 
$\ell$ is the system lag (the maximum time shift in the system equation) and $N$ is the trajectory length. 
$\theta_{lin}, \theta_{nl}$ are parameter vectors, while 
the finite set of (arbitrary) nonlinear functions $\phi_{nl}=\{\phi_{nl}^1, \dots, \phi_{nl}^K\}$ 
determines the system dynamic. 
$\sigma^i$ is the $i-$th shift operator (the sign of the superscript determines if the shift is backward or forward).	For convenience, in the paper, we omit the time dependence.
\begin{remark}
	Systems of the form \eqref{nonleq}  include  the class of SISO output-generalized bilinear systems \cite{Hemelhof}. 
\end{remark}

The extension of Definition \ref{def:hank} and Lemma \ref{lemma:gpe} to systems of the form \eqref{nonleq} reads:
\begin{definition}
	\label{def:nonlhank}
	Given a time series $w(t) \in \R^2$ and a finite set of nonlinear functions $\phi_{nl} = \{\phi_{nl}^1, \dots, \phi_{nl}^{K}\}$,  the extended Hankel matrix, $H_{L, \phi_{nl}}(w)$, is
	\begin{equation}
	\label{eq:nonlhank}
	H_{L, \phi_{nl}}(w) = 
	\begin{bmatrix}
	H_L(w) \\ H_L(\phi_{nl}^1(w)) \\ \vdots \\ H_L(\phi_{nl}^K(w))
	\end{bmatrix}.
	\end{equation}
\end{definition}

By exploiting the LTI embedding of the nonlinear dynamic proposed in \cite{ddsim-narx}, we can state the following 
\begin{lemma}
	\label{lemma:nonlgpe}
	Let $w(t)$ be a $T$-long trajectory
of a system $\mathcal{B}_{nl}$ of the form \eqref{nonleq}, and assume that
$H_{L, \phi_{nl}}(w)$ has rank $(1 + K)L + \ell$ with $L \geq \ell+ 1$. Then,
all the trajectories of $\mathcal{B}_{nl}$ are in the range of $H_{L, \phi_{nl}}(w)$.
\end{lemma}

The rank constraints in both Lemma \ref{lemma:gpe} and \ref{lemma:nonlgpe} hold as far as the data are noise-free. If we work with noise-corrupted data,  the involved Hankel matrices appear full rank. In this paper we consider white Gaussian noise on the problem data; the noise corrupted data are generated as
\begin{equation}
\label{noise}
\tilde{y}(t) = f(x_{\tilde{y}}(t), x_u(t)) + \mu r(t),
\end{equation} 
where $\mu$ is the noise level and $r(t)$ is a random scalar generated from i.i.d. standard normal distributions.

We are going to simulate future trajectories of systems of the form \eqref{nonleq} from a set of noisy data. We state the problem first, recalling some properties, and we propose an algorithm then to estimate the solution.

\section{Simulation of system trajectories}
 In this section, we first state the problem of simulating system trajectories from a set of noise-free data and present existing solution methods; we then propose an algorithm to predict the same trajectories from a set of noisy data. 
\subsection{Data-driven and model-based simulation}
\label{sec:pred}
In the behavioral setting, this problem has been studied first in \cite{ddLTIsim} for the class of LTI systems. The solution comes from Willems' fundamental lemma
\cite{flemma}. This milestone result states that it is possible to generate all the finite length system trajectories from an observed longer one.
The solution method was recently extended to a class of nonlinear systems in \cite{ddsim-narx} by a finite-dimensional  LTI embedding and a consequent 
modification of the Hankel matrix in order  to model the nonlinearities in the data.

The simulation problem we consider follows:
\begin{problem}
	\label{prob:ddsimnonl}
	Given a finite-length data trajectory $w_d=col(u_d, y_d) \in \R^2$ of an unknown nonlinear system $\mathcal{B}_{nl}$ of the form \eqref{nonleq}, a finite set of nonlinear functions $\phi_{nl}$ (that generates the system equation), an initial condition $w_{ini}=col(u_{ini}, y_{ini}) \in (\R^2)^{\ell}$ and an input signal $u_f \in \R^{t_f}$, find the output $y_f$  such that the trajectory $\begin{bmatrix}
	u_f \\ y_f
	\end{bmatrix} \in \mathcal{B}_{nl}$ and it matches the initial condition. 
\end{problem}

To deal with Problem \ref{prob:ddsimnonl}, 
 \cite{ddsim-narx} restricts to the class of generalized bilinear systems, that are systems whose functions $\phi_{nl}$ have the form $\phi_{nl}(x_y, x_u) = \psi_{nl}(x_u) + \xi_{nl}(x_u) \otimes x_y$, for some nonlinear functions $\psi_{nl}, \xi_{nl}$. The solution method relies on embedding the system $\mathcal{B}_{nl}$  into a linear one by adding sets of data transformed via the functions in $\phi_{nl}$ and exploiting a matrix representation of such special system class, where we can write the added nonlinear inputs as $u_{nl} = \psi + \begin{bmatrix}
 \Xi_p & \Xi_f
 \end{bmatrix} \begin{bmatrix}
 y_{ini} \\ y_f
 \end{bmatrix}$ (see \cite[Lemma 6]{ddsim-narx} for details).
 The restriction to such a system class involves the ability to simulate the whole trajectory by the solution of one linear system of equations. 
The solution proposed in \cite{ddsim-narx} follows:
\begin{itemize}
	\item collect all the nonlinear data terms in a new input variable $u_{nl} = \phi_{nl}(x_y, x_u)$;
	\item build the Hankel matrices 
	\begin{equation}
	\label{nonlsimhanks}
	\begin{aligned}
	H_{\ell + t_f}(u_d) &= \begin{bmatrix}
	H_{\ell}(u_d) \\ H_{t_f}(\sigma^{\ell} u_d)
	\end{bmatrix}, \\
	H_{\ell + t_f}(y_d) &= \begin{bmatrix}
	H_{\ell}(y_d) \\ H_{t_f}(\sigma^{\ell} y_d)
	\end{bmatrix}, \\
	H_{t_f}(u_{d, nl}) &= H_{t_f}(\phi_{nl}(x_{y_d}, x_{u_d})).
	\end{aligned}
	\end{equation}
\item Compute the minimum norm $g$ such that
 \begin{equation}
\label{formula16}
\begin{bmatrix}
H_{\ell}(u_d) \\ 
H_{\ell}(y_d)  \\ 
H_{t_f}(\sigma^{\ell} u_d) \\
H_{t_f}(u_{d, nl}) - \Xi_f H_{t_f}(\sigma^{\ell} y_d)
\end{bmatrix}g=\begin{bmatrix}
u_{ini} \\ y_{ini} \\  u_{f} \\ \psi+\Xi_p y_{ini}
\end{bmatrix}.
%
\end{equation}
\item Result: 
\begin{equation}
y_{f} = H_{t_f}(\sigma^{\ell} y_d) g.
\end{equation}
\end{itemize}
The solution in \eqref{formula16} has been extended to systems that allow nonlinear functions of the past outputs in \cite{Hemelhof}, that is $\phi_{nl}(x_y, x_u) = \psi_{nl}(x_y) + \xi_{nl}(x_y) \otimes x_u$. 

Beside the modern data-driven approach for the simulation of system trajectories, a classical way to solve Problem \ref{prob:ddsimnonl} is based on the  estimation of a system model. This involves the computation of a set of parameters that define the linear combination among the nonlinear system variables. 
The solution method requires to build the matrix
\begin{equation}
\label{eq:model-based}
\begin{bmatrix}
x_u(1) & \cdots & x_u(T) \\
x_y(1) & \cdots & x_y(T) \\
\phi_{nl}(x_y(1), x_u(1)) & \cdots & \phi_{nl}(x_y(T), x_u(T)) 
\\
y(1) & \cdots & y(T)
\end{bmatrix},  
\end{equation}
and then to follow these two steps:
\begin{enumerate}
	\item Compute the parameters $\hat{\theta}_{lin}, \hat{\theta}_{nl}$ as the left kernel of the matrix in \eqref{eq:model-based} (the solution is unique if the left kernel has dimension one);
	\item plug the estimated parameters in \eqref{nonleq} to recover the system equation and to compute $y$.
\end{enumerate}

\subsection{Step-by-step simulation}
A natural question is how to deal with Problem \ref{prob:ddsimnonl} whenever the data are corrupted by noise, that is how to simulate trajectories of a general system $\mathcal{B}_{nl}$ of the form \eqref{nonleq} from noisy data. 
We propose an iterative data-driven prediction algorithm that turns out to compute the same estimate as the model-based strategy.  
The algorithm exploits the link between the initial conditions and the first point of the simulated trajectory.
Finite horizon trajectories are estimated 
by iterating length-one simulations and updating the initial conditions at each step.
The problem data are preprocessed first to fulfill the rank constraint in Lemma \ref{lemma:nonlgpe}. 

We use the following result based on the truncated LQ decomposition. 
\begin{lemma}
	\label{lemma:LQ}
	Any real matrix $A$ may be decomposed as:
	\begin{equation*}
	A=LQ,
	\end{equation*}
	where $L$ is a lower triangular matrix and $Q$ is a unitary matrix\footnote{the LQ decomposition can be computed from the (more popular) QR decomposition}.
	By writing the block factorization
	\begin{equation}
	A = 	\begin{bmatrix}
	A_1 \\ A_2
	\end{bmatrix} =  \begin{bmatrix}
	L_{11} & 0 \\ L_{21} & L_{22}
	\end{bmatrix} \begin{bmatrix}
	Q_1 \\ Q_2
	\end{bmatrix} \approxeq \begin{bmatrix}
	L_{11} Q_1 \\ L_{21} Q_1
	\end{bmatrix} = A^*,
	\end{equation}
	we get that the matrix $A^*$ is a low-rank approximation of $A$. The rank reduction is the row dimension of $Q_2$. 
\end{lemma}

The proposed algorithm for the prediction of future trajectory from noisy data of the form \eqref{noise} is shown in Algorithm \ref{alg:ddsimnonl}. It uses Lemma \ref{lemma:LQ} to estimate the Hankel data matrix and iterates \eqref{formula16} to predict the points of the trajectory.
\begin{algorithm}
	\caption{Data-driven prediction for nonlinear systems}
	\label{alg:ddsimnonl}
	\begin{algorithmic}[1]
		\Require $u_d, \tilde{y}_d\  d=1, \dots, n_d$ (data points), $w_{ini}$ initial condition, $\ell$ (system lag), $T_f$ (prediction horizon), $u_f$ (future input), $\phi_{nl}$ (set of nonlinear functions) 
		\Ensure $y_f$ (predicted output) 
		
		\State Build the Hankel matrices as in \eqref{nonlsimhanks} with $t_f=1$:
			\begin{equation*}
		\begin{aligned}
	H_{\ell + 1}(u_d) &= \begin{bmatrix}
	H_{\ell}(u_d) \\ H_{1}(\sigma^{\ell} u_d)
	\end{bmatrix}, \\
	H_{\ell + 1}(\tilde{y}_d) &= \begin{bmatrix}
	H_{\ell}(\tilde{y}_d) \\ H_{1}(\sigma^{\ell} \tilde{y}_d)
	\end{bmatrix} \\
	H_{1}(\tilde{u}_{d, nl})  &= \begin{bmatrix}
	H_{1}(\phi_{nl}(x_{\tilde{y}_{d}}, x_{u_{d}}))
	\end{bmatrix}. 
	\end{aligned}
		\end{equation*}
		\State Build \begin{equation}
		\label{hd}
		H_d = \begin{bmatrix}
			H_{\ell}(u_d) \\ 
			H_{\ell}(\tilde{y}_d)  \\ 
			H_{1}(\sigma^{\ell} u_d) \\
			H_{1}(\phi_{nl}(x_{\tilde{y}_{d}},  x_{u_{d}}))
			\end{bmatrix}
		\end{equation}
		\State Use Lemma \ref{lemma:LQ} to project $H_{1}(\sigma^{\ell} \tilde{y}_d)$ onto the rows of the data matrix $H_d$ \eqref{hd} 
		\begin{equation*}
        H_{1}(\bar{y}_d) = \mathcal{P}_{H_d} (H_{1}(\sigma^{\ell} \tilde{y}_d))
		\end{equation*}
		\For{$i=1 : T_f$}
		\State Compute the vector $g_i$ of minimum norm that approximately satisfies the system  
			\begin{equation}
			\label{LSprob}
			H_d g_i \approx \underbrace{\begin{bmatrix}
			u_{ini, i} \\ y_{ini, i} \\ u_{f, i} \\ 
			\phi_{nl}(y_{ini, i}, \begin{bmatrix}
			u_{ini, i} \\ u_{f, i}
			\end{bmatrix})
			\end{bmatrix}}_{v_{ini, i}},
		\end{equation}
		where the subscript $i$ denotes the iteration step
	\State Set $\hat{y}_i = H_1(\bar{y}_d) g_i$	
		
		\State	Store the $i-$th pair $(u_i, \hat{y}_i)$
			
		\State	Set $(u_i, \hat{y}_i)$ as the last initial condition
		 \EndFor
	\end{algorithmic}
\end{algorithm}

The difference between Algorithm \ref{alg:ddsimnonl} and the approaches in \cite{ddsim-narx,Hemelhof} is that an iterative strategy imposes no restriction on the system class. 
We are going to show how Algorithm \ref{alg:ddsimnonl} strongly connects the model-based approach to the behavioral theory (at least for the considered prediction problem).

\begin{theorem}
	\label{th:equivalentestimators}
	The iterative data-driven prediction of Algorithm \ref{alg:ddsimnonl}  estimates the same solution as the model-based approach. 
\end{theorem}
\begin{proof}
	We show that one step of Algorithm \ref{alg:ddsimnonl} and one step of the model-based prediction (briefly recalled in the deterministic case in Section \ref{sec:pred}), starting from the same data, initial conditions and \emph{future} input signal, estimate the same point.  The result holds true by assuming that we use the same set of basis functions $\phi_{nl}$ and that the data matrix is approximated by the LQ decomposition (see Lemma \ref{lemma:LQ}). 
	
	\textbf{Model-based:}
	Following \eqref{eq:model-based}, 
	we need to estimate the system parameters $\theta_{lin}, \theta_{nl}$ as a (approximate) left null space of the matrix
	\begin{equation}
\underbrace{
\begin{bmatrix}
x_u(1) & \cdots & x_u(T) \\
x_{\tilde{y}}(1) & \cdots & x_{\tilde{y}}(T) \\
\phi_{nl}(x_{\tilde{y}}(1), x_u(1)) & \cdots & \phi_{nl}(x_{\tilde{y}}(T), x_u(T)) 
\\
\tilde{y}(1) & \cdots & \tilde{y}(T)
\end{bmatrix}}_{\mathcal{S}}.
\end{equation}

	
Since the outputs are noisy, 
	the matrix $\mathcal{S}$ is full rank. We impose the existence of a left kernel by projecting the last row on the space generated by the data $x_u, x_{\tilde{y}}$ and their nonlinear transformations. This is done using Lemma \ref{lemma:LQ}:
	\begin{equation}
	\label{eq:lq}
	\mathcal{S} =
	 \begin{bmatrix}
	L_{11} & 0 \\ L_{21} & L_{22}
	\end{bmatrix} \begin{bmatrix}
	Q_1 \\ Q_2
	\end{bmatrix}.
	\end{equation}
	The model-based estimator is then given by the projection of the row $\tilde{y}$ onto the space generated by the (noisy) rows $x_u, x_{\tilde{y}}, \phi_{nl}(x_{\tilde{y}}, x_u)$ by neglecting the term $L_{22} Q_2$ in \eqref{eq:lq}:
	\begin{equation*}
	\hat{y}_{MB} = \mathcal{P}_{x_u, x_{\tilde{y}}, \phi_{nl}(x_{\tilde{y}}, x_u)}\  (\tilde{y}) = L_{21} Q_1  = L_{21} L_{11}^{-1} L_{11} Q_1.
	\end{equation*}
	If we set $[\hat{\theta}_{lin}, \hat{\theta}_{nl}]  = L_{21} L_{11}^{-1}$, 
	 the one-step model-based estimator can be written as
	\begin{equation}
	\hat{y}_{MB} = [\hat{\theta}_{lin}, \hat{\theta}_{nl}]  v_{ini, i},
	\tag{MB}
	\end{equation}
	where $v_{ini, i}$ is the vector in the right-hand side of \eqref{LSprob}.  
	
	\textbf{Data-driven (Algorithm \ref{alg:ddsimnonl}-one step):}
	the key observation is that the matrix $\mathcal{S}$ in \eqref{eq:model-based} and the block Hankel matrix $\begin{bmatrix}
	H_d \\ 
	H_{1}(\sigma^{\ell} \tilde{y}_d)
	\end{bmatrix}$ are the same, up to a permutation of the rows (this happens since $t_f=1$). 
	Following Algorithm \ref{alg:ddsimnonl}, 
	any length-one data-driven prediction can be computed as follows: 
	\begin{enumerate}
		\item Project $H_1(\sigma^{\ell} \tilde{y}_d)$ onto the rows generated by the data (see line 3 of Algorithm \ref{alg:ddsimnonl}) and get $\bar{y} = H_1(\bar{y})$;
		\item compute $g_i$ that solves the least squares problem in \eqref{LSprob};
		\item set $\hat{y}_{DD} = \bar{y} g_i$.
	\end{enumerate}
	Since the vector $g_i$ defines a linear combination of the columns of the data matrix, because of the projection in the first step, there exists $\gamma$ such that $g_i = Q_1^T \gamma$ (where $Q_1$ is the matrix in \eqref{eq:lq}) \cite{BCFautomatica}. Doing so, we get 
	\begin{equation} 
	\label{eq:gamma}
	\begin{aligned}
	v_{ini, i} &\approx 	H_d
	 Q_1^T \gamma
	 = L_{11} Q_1 Q_1^T \gamma =L_{11} \gamma. 
	 \end{aligned}
	\end{equation}
	The second block equation in \eqref{eq:lq} gives
	\begin{equation}
	\label{eq:yDD}
	\hat{y}_{DD} = \bar{y} g_i = (L_{21} Q_1 + L_{22} Q_2) g_i = (L_{21} Q_1) Q_1^T \gamma  = L_{21} \gamma,
	\end{equation}
	because of the orthogonality between $Q_1$ and $Q_2$ (that follows from the projection step obtained with the approximate LQ decomposition). 
	The thesis holds true since $\gamma= L_{11}^{-1} v_{ini, i}$ (see  \eqref{eq:gamma}). 
\end{proof}

A series of remarks about Algorithm \ref{alg:ddsimnonl} follows:
\begin{itemize}
	\item \textbf{Deterministic setting} In the deterministic setting, the solution computed by Algorithm \ref{alg:ddsimnonl} matches the correct output for the systems of the form \eqref{nonleq}. In particular, this holds for the output generalized bilinear systems presented in \cite{Hemelhof};
	\item \textbf{Model-based vs data-driven} Theorem \ref{th:equivalentestimators} shows an equivalence between the model-based and the behavioral approach for the prediction of future trajectories. This comes from the computational choices in Algorithm \ref{alg:ddsimnonl};
	\item {\textbf{Computational cost}}
	The computations in Algorithm \ref{alg:ddsimnonl} require the solution of $T_f$ least squares problems of the form \eqref{LSprob}.  The computational cost can be reduced by observing that all the linear systems have the same data matrix and only the right hand side (initial conditions and input) changes at each iteration \cite{Golubvanloan}. 
\end{itemize}

A natural question arises: is it still worth adopting an iterative computational strategy whenever it is possible to simulate the trajectory all at once (that is, \textit{e.g.},
 for the class of output-generalized bilinear systems in \cite{Hemelhof})? 
Our goal is not to answer this question, but 
we simply proposed a way to predict the trajectories of an estimated system, by  shortening the distance between two apparently different approaches.

\section{Connection between behaviors and system models}
\label{sec:dd-vs-mb}
The common characterization of data-driven algorithms is that they do not need to write down explicitly a system model (a set of equations) to develop a solution method; a set of data is enough to learn the system dynamics. This is a feature of the behavioral theory too, where systems are represented (in a data-driven fashion) by rank constrained  block Hankel matrices built from observed trajectories.
However, this is simply a different way of writing down a system, since the rank constraint on the Hankel matrix is equivalent to the existence of a system representation (an equation whose coefficients are given by the left  kernel of the matrix), that is \emph{hidden} in the left kernel of such a matrix. 
Thus, there are two different but interconnected ways of approaching the same problem. 
 Both the approaches can be valid alternatives, and studying their properties and connections (mainly in the field of control problems) is a research topic by itself \cite{chiuso2023harnessing,DorflerTesiDepersis,dirVSindir,bridging}. 

In the behavioral theory (despite being a data-driven approach), system representations exist and they can be useful.  The rank constraints on the Hankel matrices enforce the existence of a system representation (commonly known as kernel representation), and vice versa. Indeed, 
 the definition of a dynamical system $\mathcal{B}$, in general, involves a difference equation defined by a polynomial operator $R$:
\begin{equation}
\label{beh}
\mathcal{B} = \Big\{ w \in \R^q | R(\sigma)
w(t) = 0 \Big\}, \ t=1, \dots, T,
\end{equation}
where $\sigma(w(t)) = w(t+1)$ is the shift operator. 
We read in \eqref{beh} that the trajectories $w$ of the system $\mathcal{B}$ and its (non-unique) representation $R$, given by the finite-length snapshots of $w(t)$ and the rows of the operator $R$, respectively, are two \emph{distinct} but \emph{interconnected} objects. 
They represent the data-based and the model-based approach to the same problem, respectively, and they are linked by \eqref{beh}. 

We can look at some examples 
from the literature, where the same problem is solved by these two different approaches (restricted to the case of LTI behaviors):
	\begin{itemize}
		\item the distance between LTI systems: \cite{bdistanceeth} is based on the Hankel matrices generated by the two system trajectories, while
		 \cite{distECC} or \cite{Anglesdemoor} are based on the systems representations only (being representation invariant, though);
		 \item controllability test: \cite{flemmaunc} is based on Hankel matrices, while \cite{cdc19} relies on a coprimeness test between the polynomials in the system representations. 
	\end{itemize}

The row dimension of $R$ plays an important role in the prediction problem, and in particular in the length of the to-be-computed trajectory. $R$ is said \textit{minimal} if its rows are linearly independent \cite{Willems2}. Each linearly dependent row in $R$ increases by one the length of the predicted trajectory. Thus, increasing the row size of the Hankel matrix can be viewed as adding (linearly dependent) equations in a vectorial system model. Observe that, in this last case, the rank of the Hankel matrix does not change (but the dimension of its null space does).

\section{Examples}
In the following, we are going to run and analyze some examples of estimation of future trajectories computed by Algorithm \ref{alg:ddsimnonl}. The analysis will focus on different perturbations on the same set of noise-free data, to observe and discuss how the results change for a given set of nonlinear functions $\phi_{nl}$. 

The system we consider is given by the difference equation
\begin{equation}
\label{sys}
y(t+2) = \sin(y(t+1)) -.1y(t)^2 + u(t).
\end{equation}
whose lag is $\ell=2$. 
The setup for the experiment follows:
\begin{enumerate}
	\item build a trajectory of $80$ points for the system \eqref{sys} starting from $(1, 1)$ and a random input signal; 
	
	\item length-$68$ problem data: generate $100$ random input signals and the corresponding noisy outputs as in \eqref{noise} with $\mu= 0.1$;
	
	\item initial conditions: $\begin{pmatrix}
	u_{69} \\ u_{70}
	\end{pmatrix}, \begin{pmatrix}
	y_{69} \\ y_{70}
	\end{pmatrix}$ from point 1 (they are the same in all the $100$ experiments);

	\item for each of the $100$ data pairs $(u, \tilde{y})$ from point 2,  use the given set of basis functions, the initial conditions from point 3 and
	the last ten inputs from point 1 
	to predict a trajectory (by using
	 Algorithm \ref{alg:ddsimnonl}). 
\end{enumerate}
\begin{remark}
	The noise-free data (first $68$ inputs and outputs) in point 1 are not useful, but we used them only to generate the initial conditions (and an exact, but unknown, trajectory to be compared with the results of the experiment).
\end{remark}
The goal of the experiment is to observe and analyze the results of the statistics over the $100$ solutions computed by Algorithm \ref{alg:ddsimnonl}, comparing them with the true system trajectory. To make the example more realistic, we consider several nonlinear functions to (hopefully) include the ones in the system equation (or to well approximate them, at least). In particular, we consider
$$
\phi_{nl} = \{ y^2, y^3, y^4, \cos y, \sin y, e^y\}.
$$ 

\begin{remark}
	In the deterministic setting, the addition of linearly independent functions to the set $\phi_{nl}$ does not change the result of the prediction problem. This is because
	\begin{itemize}
		\item in the system equation, the corresponding parameter is set to zero;
		\item in the block-Hankel matrix representation, while the row dimension of the matrix increases, the rank loss is the same. 
	\end{itemize}
  For this reason, we did not add the input $u$ as an argument of the nonlinear functions in $\phi_{nl}$. 
\end{remark}

Considering a bigger set of functions $\phi_{nl}$ than the minimal one in the stochastic setting has two main consequences:
\begin{enumerate}
	\item the noise spreads through more (nonlinear) terms;
	\item bigger block-Hankel matrices usually means ill-conditioned problems. 
\end{enumerate}
Estimating the system from a set of functions $\phi_{nl}$ bigger than the minimal one means \emph{choosing} a particular system among a bigger model class defined by all the considered basis functions. The addition of a regularization term in the solution of the least squares problem  \eqref{LSprob}  enforces the identification of the sought system.

The regularization in system identification is a classical approach in the literature  \cite{pillonetto2022regularized}, and different choices for the norm of the regularization term 
 already appeared in system identification problems \cite{med2,Bruntonsindy,LSSVM}.
 We consider the following two choices:
\begin{equation}
\label{eq:reg1}
\min_{g_i} \| H_d g_i - v_{ini, i} \|_2 + \lambda_1 \| g_i \|_1,
\end{equation}
\begin{equation}
\label{eq:reg2}
\min_{g_i} \| H_d g_i - v_{ini, i} \|_2 + \lambda_2 \| g_i \|_2,
\end{equation}
where $\lambda_1$ and $\lambda_2$ are regularization parameters. In the following experiments, we solve all the $100$ problems for different values of $\lambda_1$ and $\lambda_2$ and we choose the one for which the average over the $100$
estimated trajectories is closer (in norm) to the exact one. 

The results for the two choices are shown in Figures \ref{fig:reg1} and \ref{fig:reg2}, respectively. 
\begin{figure}[htb]
	\centering
	\includegraphics[width=8cm,height=5cm]{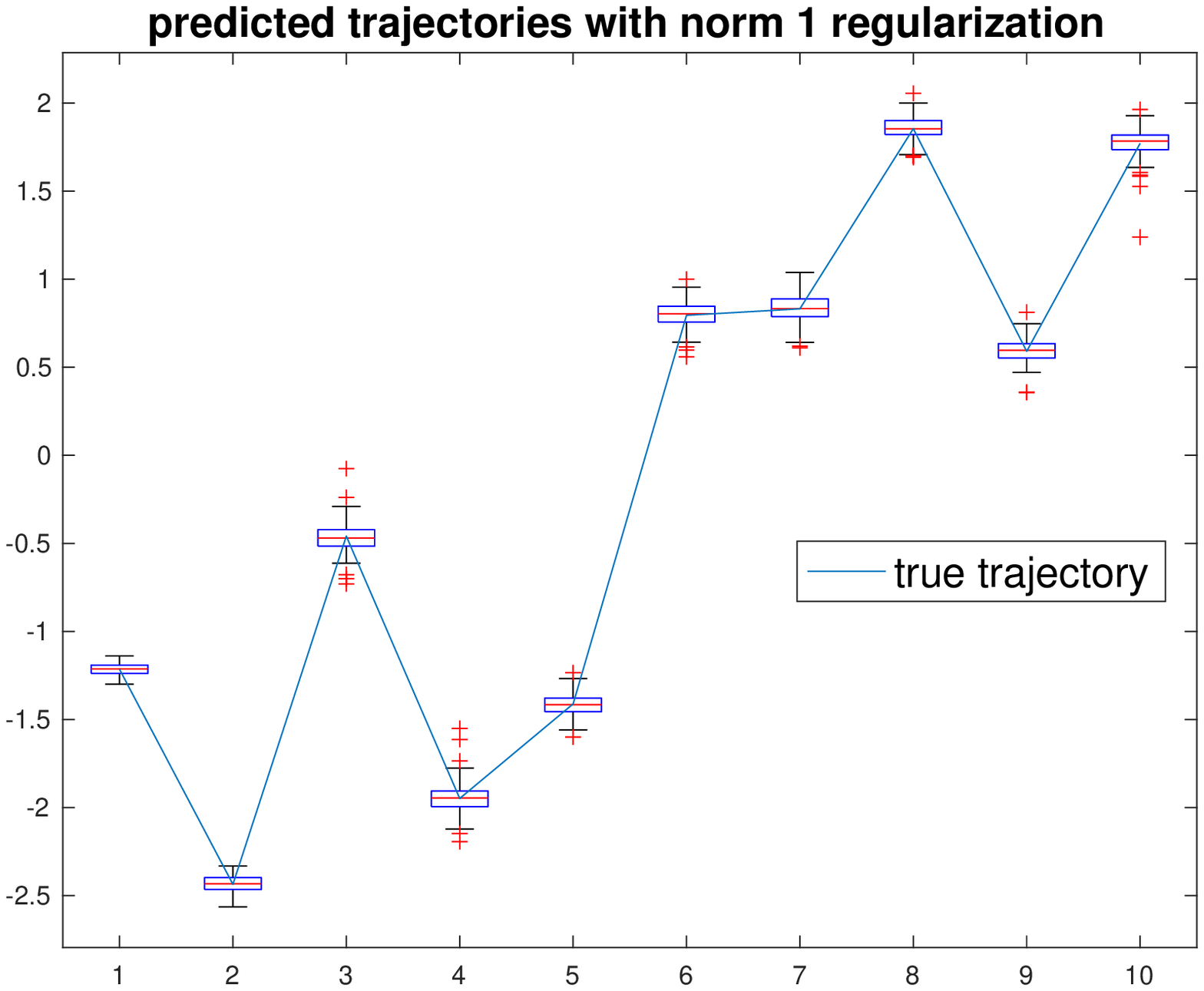}
	\caption{\label{fig:reg1}Estimation of a trajectory of \eqref{sys} from noisy data and a bigger set of basis functions: norm-1 regularization.}
\end{figure}
\begin{figure}[htb]
	\centering
	\includegraphics[width=8cm,height=5cm]{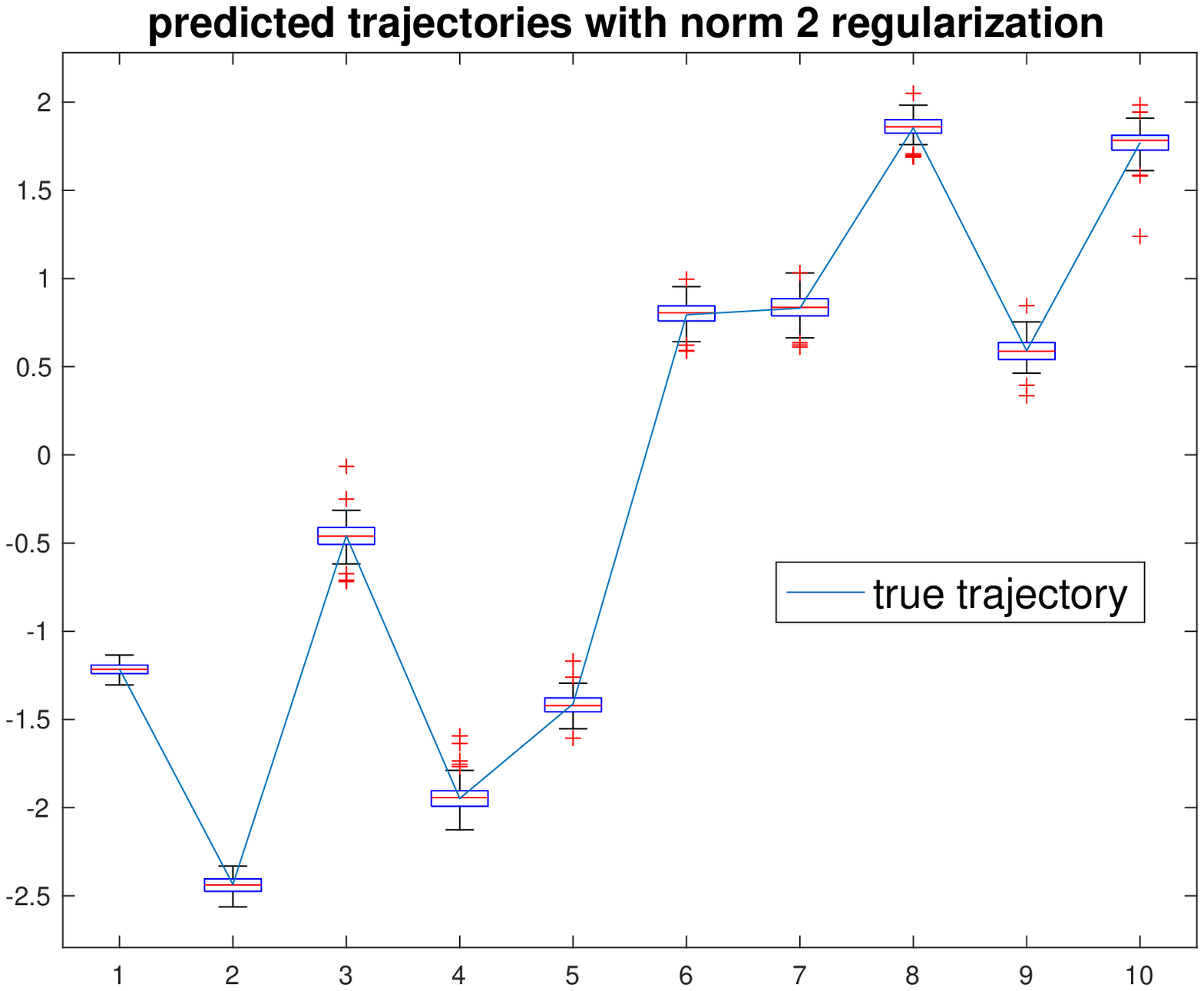}
	\caption{\label{fig:reg2}Estimation of a trajectory of \eqref{sys} from noisy data and a bigger set of basis functions: norm-2 regularization.}
\end{figure}

Both experiments show that most of the predicted trajectories are very close to the exact (but unknown) one.  Indeed, the two figures look similar.  The choice between \eqref{eq:reg1} and \eqref{eq:reg2} corresponds to a different weighting of the entries of the solution vector. 
The advantage of \eqref{eq:reg2} can be the easier implementation.

 To key points of the experiments follow:
 \begin{enumerate}
 	\item the choice of the basis functions in $\phi_{nl}$ influences the accuracy of the computed trajectories (that is, how well they approximate the exact solution). The preliminary knowledge of such functions is a strong assumption;
 	\item the addition of regularization in the solution method can deal with an increasing number of basis functions as well as with ill-conditioned problems; 
 	\item the choice of the norm in the regularization can influence the accuracy of the estimated solution (this did not happen   In the considered example because of the small problem size and the low noise level). 
 \end{enumerate}

\section{Applications in control}
As well as being important itself, the proposed prediction algorithm can be the key tool in the solution of some nonlinear control problems in the stochastic setting.  We briefly sketch some ideas in the following. 
\subsection{Output matching problem}
The output matching problem for output-generalized bilinear systems has been studied in \cite{Hemelhof} in the deterministic setting. Roughly speaking, such a problem consists in the analogous of a prediction problem, where the role of inputs and outputs is switched. So we aim to compute the inputs that generated a given output trajectory. 

The extension of this problem to the stochastic setting can be straightforward by using the prediction algorithm proposed in the paper. 

\subsection{Nonlinear predictive control}
\label{sec:ddpc}
The estimation of future system trajectories is the cornerstone of predictive control strategies. This approach consists of iterating the following steps:
\begin{enumerate}
	\item predict N future steps from the given set of data;
	\item optimize the estimated steps with respect to a cost function;
	\item get an optimal controller from the optimization process;
	\item keep the first input value from the controller and go to the next iteration. 
\end{enumerate}

If the system is represented by a (estimated) equation, this problem is commonly known as MPC (Model Predictive Control). 
But, it is possible to solve the same problem by replacing the system equation with a set of (noisy) data. 
The data-driven prediction algorithm proposed in the paper can be used for the first step. 
 Optimization strategies for the data-driven setting are currently object of research.

\section{Conclusion}
\label{sec:conc}
Moving on the line between model-based and data-driven approaches for the theory of nonlinear systems, we proposed an iterative data-driven algorithm that predicts the same solution as the model-based approach from a set of noisy data. The goal was twofold: show some connection between the two (apparently) different strategies and extend some existing results from the deterministic to the stochastic setting.

The proposed prediction algorithm can be useful in the solution of some nonlinear control problems. Future work can focus on the analysis and development of nonlinear control methods (in particular, the one described in Section \ref{sec:ddpc}).

\section*{ACKNOWLEDGMENT}

A. Fazzi is member of the Gruppo Nazionale Calcolo Scientifico-Istituto Nazionale di Alta Matematica (GNCS-INdAM).


\bibliographystyle{IEEEtran}
\bibliography{Nonlpred.bib,mypapers.bib}

\end{document}